\documentclass[fleqn]{zzz}
\usepackage{color}
\usepackage{tikz}
\usepackage{amsmath,amsthm,amsbsy,amssymb}
\usepackage{graphicx}
\usepackage{rotating}
\usepackage{fixmath}
\usepackage[pdftex]{hyperref}
\usepackage{soul}
\usepackage{bm}
\usepackage{mathrsfs}

\hypersetup{
 colorlinks = true,
 urlcolor = blue,
 linkcolor = red,
 citecolor = green,
}

\definecolor{Mycolor2}{HTML}{0f4c5c}
\definecolor{Mycolor1}{HTML}{7b2cbf}
\definecolor{Mycolor3}{HTML}{ff6d00}

\sethlcolor{black}

\newcommand\boro[1]{{\textcolor{black}{#1}}}

\newcommand{\hyp}[5]{\,\mbox{}_{#1}F_{#2}\!\left(
\genfrac{}{}{0pt}{}{#3}{#4};#5\right)}

\newcommand{\Ohyp}[5]{\,\mbox{}_{#1}{\bm{F}}_{#2}\!\left(
\genfrac{}{}{0pt}{}{#3}{#4};#5\right)}

\def\cprime{$'$}

\newtheorem{thm}{Theorem}[section]
\newtheorem{cor}[thm]{Corollary}

\newtheorem{rem}[thm]{Remark}
\newtheorem{lem}[thm]{Lemma}

\makeatletter
\def\eqnarray{\stepcounter{equation}\let\@currentlabel=\theequation
\global\@eqnswtrue
\tabskip\@centering\let\\=\@eqncr
$$\halign to \displaywidth\bgroup\hfil\global\@eqcnt\z@
$\displaystyle\tabskip\z@{##}$&\global\@eqcnt\@ne
\hfil$\displaystyle{{}##{}}$\hfil
&\global\@eqcnt\tw@ $\displaystyle{##}$\hfil
\tabskip\@centering&\llap{##}\tabskip\z@\cr}

\def\endeqnarray{\@@eqncr\egroup
\global\advance\c@equation\m@ne$$\global\@ignoretrue}

\makeatother

\newcommand{\N}{{\mathbb N}}

\newcommand{\C}{{\mathbb C}}

\newcommand{\dd}{{\mathrm d}}

\usepackage{scalerel,url}
\let\svus_
\catcode`_=\active %
\def_{\ifmmode\expandafter\svus\expandafter\bgroup\expandafter\lowerit
\else\svus\fi}
\def\lowerit#1{\ThisStyle{\raisebox{-2\LMpt}{$\SavedStyle#1$}}\egroup}
\makeatletter
 \AtBeginDocument{
 \check@mathfonts
 \fontdimen13\textfont2=3.5pt
 \fontdimen14\textfont2=3.5pt
 \fontdimen16\textfont2=2.5pt
 \fontdimen17\textfont2=2.5pt
 }
\makeatother

\begin{document}

\renewcommand{\PaperNumber}{***}

\FirstPageHeading

\ShortArticleName{Multi-integral representations for Jacobi functions}

\ArticleName{Multi-integral representations for\\ Jacobi functions of the 
first and second kind}

\Author{Howard S. Cohl
$\,^{\dag}\orcidB{}$ and 
Roberto S. Costas-Santos$\,^{\S}\orcidA{}$
}

\AuthorNameForHeading{H.~S.~Cohl and R.~S.~Costas-Santos}
\Address{$^\dag$ Applied and Computational 
Mathematics Division, National Institute of Standards 
and Tech\-no\-lo\-gy, Mission Viejo, CA 92694, USA
\URLaddressD{
\href{http://www.nist.gov/itl/math/msg/howard-s-cohl.cfm}
{http://www.nist.gov/itl/math/msg/howard-s-cohl.cfm}
}
} 
\EmailD{howard.cohl@nist.gov} 

\Address{$^\S$ Department of Quantitative Methods, Universidad Loyola Andaluc\'ia, Sevilla, Spain
} 
\URLaddressD{
\href{http://www.rscosan.com}
{http://www.rscosan.com}
}
\EmailD{rscosa@gmail.com} 


\ArticleDates{Received \today~in final form ????; Published online ????}

\Abstract{{One may consider the generalization of Jacobi polynomials 
and the Jacobi function of the second kind to a general function where the 
index is allowed to be a complex number instead of a non-negative integer. 
These functions are referred to as Jacobi functions. In a similar fashion 
as associated Legendre functions, these break into two categories, functions 
which are analytically continued from the real line segment $(-1,1)$ and 
those continued from the real ray $(1,\infty)$.
Using properties of Gauss hypergeometric functions,
we derive multi-derivative and multi-integral representations for
the Jacobi functions of the first and second kind.}
}

\Classification{33C45;
44A20; 42C05; 33C05; 
}

\vspace{-0.5cm}
\section{{Preliminaries}}
{
Throughout this paper we adopt the following set notations:
$\mathbb N_0:=\{0\}\cup\mathbb N=\{0, 1, 2, 3, \ldots\}$, and
we use the set $\mathbb C$ which represents the complex
numbers.
The Pochhammer symbol for $a\in\mathbb C$,
$n\in\mathbb N_0$ 
is given by \cite[
\href{http://dlmf.nist.gov/5.2.E4}{(5.2.4)}, \href{http://dlmf.nist.gov/5.2.E5}{(5.2.5)}]{NIST:DLMF}
$
(a)_n:=(a)(a+1)\cdots(a+n-1).
$
The gamma function \cite[\href{http://dlmf.nist.gov/5}{Chapter 5}]{NIST:DLMF} is related to 
the Pochhammer symbol, namely
for $a\in\mathbb C\setminus-{\mathbb N}_0$, one
 has
\begin{equation}
(a)_n=\frac{\Gamma(a+n)}{\Gamma(a)},
\label{Pochdef}
\end{equation}
which allows one to extend the definition to non-positive 
integer values of $n$.
Some other properties of Pochhammer symbols which we will use are
($n,k\in\N_0$, $n\ge k$)
\begin{eqnarray}
&&\hspace{-10cm}\Gamma(a-n)=\frac{(-1)^n\Gamma(a)}{(1-a)_n}\label{Poch1}\\
&&\hspace{-10cm}(-n)_k=\frac{(-1)^kn!}{(n-k)!}.\label{Poch2}
\end{eqnarray}
One also has the following
expression for the generalized 
binomial coefficient 
for $z\in\C$, $n\in\N_0$
\cite[\href{http://dlmf.nist.gov/1.2.E6}{(1.2.6)}]{NIST:DLMF}
\begin{equation}
\label{binomz}
\binom{z}{n}=\frac{(-1)^n (-z)_n}{n!}.
\end{equation}
We will also use the common notational 
product convention,
$a_l\in\C$, $l\in\N$, $r\in\N_0$, e.g.,
\begin{eqnarray}
&&\hspace{-7.7cm}(a_1,\ldots,a_r)_k:=(a_1)_k(a_2)_k\cdots(a_r)_k,\\[0.1cm]
&&\hspace{-7.7cm}\Gamma(a_1,\ldots,a_r):=\Gamma(a_1)\cdots\Gamma(a_r).
\end{eqnarray}
}

{
The generalized hypergeometric
function \cite[\href{http://dlmf.nist.gov/16}{Chapter 16}]{NIST:DLMF} is defined by the 
infinite series \cite[\href{http://dlmf.nist.gov/16.2.E1}{(16.2.1)}]{NIST:DLMF}
\begin{equation} \label{genhyp}
\hyp{r}{s}{a_1,\ldots,a_r}
{b_1,\ldots,b_s}{z}
:=
\sum_{k=0}^\infty
\frac{(a_1,\ldots,a_r)_k}
{(b_1,\ldots,b_s)_k}
\frac{z^k}{k!},
\end{equation}
where $b_j\not \in -\mathbb N_0$, for 
$j\in\{1, \dots, s\}$;
and elsewhere by analytic continuation.
Further define Olver's (scaled or regularized) 
generalized hypergeometric series
\begin{equation}
\Ohyp{r}{s}{a_1,\ldots,a_r}{b_1,\ldots,b_s}{z}:=
\frac{1}{\Gamma(b_1,\ldots,b_s)}\hyp{r}{s}{a_1,\ldots,a_r}{b_1,\ldots,b_s}{z}=
\sum_{k=0}^\infty \frac{(a_1,\ldots,a_r)_k}{\Gamma(b_1\!+\!k,\ldots,b_s\!+\!k)}
\frac{z^k}{k!},
\end{equation}
which is entire for all $a_l,b_j\in\C$, $l\in\{1,\ldots,r\}$, $j\in\{1,\ldots,s\}$.
Both the generalized and Olver's generalized hypergeometric series, if 
nonterminating, are entire if $r\le s$, convergent for $|z|<1$ if $r=s+1$ 
and divergent if $r\ge s+1$.
}
\medskip

The Gauss hypergeometric function ${}_2F_1$, first studied by Gauss, has many
useful and interesting properties. For instance it
satisfies several useful derivative relations
which we will rely upon. These are given
by 
cf.~\cite[\href{http://dlmf.nist.gov/15.5.E2}{(15.5.2)}, \href{http://dlmf.nist.gov/15.5.E4}{(15.5.4)}, \href{http://dlmf.nist.gov/15.5.E6}{(15.5.6)}, \href{http://dlmf.nist.gov/15.5.E9}{(15.5.9)}]{NIST:DLMF}
\begin{eqnarray}
&&\hspace{-0.2cm}\frac{\dd^n}{\dd w^n}\Ohyp21{a,b}{c}{w}
=(a)_n(b)_n\Ohyp21{a\!+\!n,b\!+\!n}{c\!+\!n}{w},
\label{der2}
\\[0.1cm]
&&\hspace{-0.2cm}\frac{\dd^n}{\dd w^n}w^{c-1}\Ohyp21{a,b}{c}{w}
=
w^{c-n-1}\Ohyp21{a,b}{c\!-\!n}w,
\label{der4}
\\[0.1cm]
&&\hspace{-0.2cm}\frac{\dd^n}{\dd w^n}(1-w)^{a+b-c}\Ohyp21{a,b}{c}{w}
=
(c-a)_n(c-b)_n
(1-w)^{a+b-c-n}\Ohyp21{a,b}{c\!+\!n}w,
\label{der6}
\\[0.1cm]
&&\hspace{-0.2cm}\frac{\dd^n}{\dd w^n}w^{c-1}(1-w)^{a+b-c}
\Ohyp21{a,b}{c}{w}
=w^{c-n-1}(1-w)^{a+b-c-n}\Ohyp21{a\!-\!n,b\!-\!n}{c\!-\!n}{w}.
\label{der9}
\end{eqnarray}
The Jacobi polynomial and many of the other
special functions which we will study in
this paper are given in terms of a terminating
Gauss hypergeometric function.
In particular, one of the most important 
classical orthogonal polynomials, the Jacobi polynomial, is defined 
as \cite[\href{http://dlmf.nist.gov/18.5.E7}{(18.5.7)}]{NIST:DLMF}
\begin{equation}
\label{Jacobipolydef}
P_n^{(\alpha,\beta)}(x):=
\frac{(\alpha+1)_n}{n!}
\hyp21{-n,n+\alpha+\beta+1}{\alpha+1}{\frac{1\!-\!x}{2}}.
\end{equation}

\section{{Definitions and properties of the Jacobi functions}}
\label{defJac}

Jacobi functions are
complex solutions $w=w(z)=w_\gamma^{(\alpha,\beta)}(z)$ 
to the Jacobi differential equation
\cite[\href{http://dlmf.nist.gov/18.8.T1}{Table 18.8.1}]{NIST:DLMF}
\begin{equation}
(1-z^2)\frac{\dd^2w}{\dd z^2}+\left(\beta-\alpha-z(\alpha+\beta+2)\right)
\frac{\dd w}{\dd z}
+\gamma(\alpha+\beta+\gamma+1)w=0,
\label{Jacde}
\end{equation}
which is a second order linear homogeneous differential equation.
Jacobi functions of the first and second kind are solutions to the Jacobi differential equation \eqref{Jacde} which are regular as $|z|\to1$ and $|z|\to\infty$ respectively. 
There are several important references for
Jacobi functions such as 
\cite[Section 10.8]{ErdelyiHTFII},
and 
\cite{Durand78,Durand79,FlenstedJensenKoornwinder73,Koornwinder84,mr2149265,Szego,WimpMcCabeConnor97}.


\subsection{The Jacobi function of the first kind}
The Jacobi function of the first kind
is a generalization of the Jacobi polynomial where the degree is no longer 
restricted to be an integer. In the following section we provide some important properties 
for the Jacobi function of the first kind.
\subsubsection{Single Gauss hypergeometric representations}
\boro{In the following result we present the four 
single Gauss hypergeometric function representations 
of the Jacobi function of the first kind.}

\begin{thm}
\label{Firstthm}
Let
$\alpha,\beta,\gamma\in\C$ such that
$\alpha+\gamma\not\in-\N$. Then the Jacobi function
of the first kind 
$P_\gamma^{(\alpha,\beta)}:\C\setminus(-\infty,-1]\to\C$ 
can be defined by
\begin{eqnarray}
&&\hspace{-0.0cm}P_\gamma^{(\alpha,\beta)}(z)=
\frac{\Gamma(\alpha+\gamma+1)}{
\Gamma(\gamma+1)}
\Ohyp21{-\gamma,\alpha+\beta+\gamma+1}
{\alpha+1}{\frac{1-z}{2}}
\label{Jac1}\\
&&\hspace{1.6cm}=\frac{\Gamma(\alpha+\gamma+1)}{
\Gamma(\gamma+1)}
\left(\frac{2}{z+1}\right)^\beta\Ohyp21
{-\beta-\gamma,\alpha+\gamma+1}{\alpha+1}{\frac{1-z}{2}}\label{Jac2}\\
&&\hspace{1.6cm}=\frac{\Gamma(\alpha+\gamma+1)}
{
\Gamma(\gamma+1)}
\left(\frac{z+1}{2}\right)^\gamma\Ohyp21{-\gamma,-\beta-\gamma}
{\alpha+1}{\frac{z-1}{z+1}}\label{Jac3}\\
&&\hspace{1.6cm}=\frac{\Gamma(\alpha+\gamma+1)}
{
\Gamma(\gamma+1)}
\left(\frac{2}{z+1}\right)^{\alpha+\beta+\gamma+1}\Ohyp21
{\alpha+\gamma+1,\alpha+\beta+\gamma+1}{\alpha+1}{\frac{z-1}{z+1}}.\label{Jac4}
\end{eqnarray}
\label{bigPthm}
\end{thm}

\begin{proof}
Start with \eqref{Jacobipolydef} and replace the Pochhammer
symbol by a ratio of gamma functions using \eqref{Pochdef}
and the factorial $n!=\Gamma(n+1)$ and substitute $n\mapsto\gamma\in\C$, 
$x\mapsto z$. Application of Pfaff's $(z\mapsto z/(z-1))$ and Euler's 
$(z\mapsto z)$ transformations \cite[\href{http://dlmf.nist.gov/15.8.E1}{(15.8.1)}]{NIST:DLMF} provides the 
other three representations. This completes the proof.
\end{proof}
One of the consequences of the definition of the 
Jacobi function of the first kind is the following
special value:
\begin{equation}
P_\gamma^{(\alpha,\beta)}(1)=\frac{\Gamma(\alpha\!+\!\gamma\!+\!1)}
{\Gamma(\alpha\!+\!1)\Gamma(\gamma\!+\!1)},
\label{Jacone}
\end{equation}
where $\alpha+\gamma\not\in-\N$.

\subsection{{The Jacobi function of the second kind}}
Studies of the Jacobi function of the second kind $Q_\gamma^{(\alpha,\beta)}$ traditionally used a degree $\gamma$ which was integer valued (see for instance \cite[\S4.61]{Szego}). However, in this paper we treat
the Jacobi function of the second kind where $\gamma$ is not necessarily restricted to be an integer.
In the following material we
provide some important properties for the Jacobi function of the second kind.
\subsubsection{Single Gauss hypergeometric representations}
\boro{Below we give the four single Gauss hypergeometric function representations 
of the Jacobi function of the second kind.}
{
\begin{thm}\label{thmQ}
Let $\gamma,\alpha,\beta,z\in\C$ such that 
$z\in\C\setminus[-1,1]$,
$\alpha+\gamma,\beta+\gamma\notin-\N$.
Then, the Jacobi function of the second kind
has the following Gauss hypergeometric representations
\begin{eqnarray}
&&\hspace{-0.2cm}Q_\gamma^{(\alpha,\beta)}(z) :=
\frac{2^{\alpha+\beta+\gamma}\Gamma(\alpha+\gamma+1)\Gamma(\beta+\gamma+1)}
{(z-1)^{\alpha+\gamma+1}(z+1)^\beta}
\Ohyp21{\gamma+1,\alpha+\gamma+1}{\alpha+\beta+2\gamma+2}
{\frac{2}{1-z}}
\label{dJsk1}\\[0.2cm]
&&\hspace{1.5cm}=
\frac{2^{\alpha+\beta+\gamma}\Gamma(\alpha+\gamma+1)\Gamma(\beta+\gamma+1)}
{(z-1)^{\alpha+\beta+\gamma+1}}
\Ohyp21{\beta+\gamma+1,\alpha+\beta+\gamma+1}{\alpha+\beta+2\gamma+2}
{\frac{2}{1-z}}\label{dJsk4}\\[0.2cm]
&&\hspace{1.5cm}=
\frac{2^{\alpha+\beta+\gamma}\Gamma(\alpha+\gamma+1)\Gamma(\beta+\gamma+1)}
{(z-1)^{\alpha}(z+1)^{\beta+\gamma+1}}
\Ohyp21{\gamma+1,\beta+\gamma+1}{\alpha+\beta+2\gamma+2}
{\frac{2}{1+z}}\label{dJsk3}\\[0.2cm]
&&\hspace{1.5cm}=
\frac{2^{\alpha+\beta+\gamma}\Gamma(\alpha+\gamma+1)\Gamma(\beta+\gamma+1)}
{(z+1)^{\alpha+\beta+\gamma+1}}
\Ohyp21{\alpha+\gamma+1,\alpha+\beta+\gamma+1}{\alpha+\beta+2\gamma+2}
{\frac{2}{1+z}}.
\label{dJsk2}
\end{eqnarray}
\end{thm}
}
{
\begin{proof}
Start with \cite[(10.8.18)]{ErdelyiHTFII} and
let $n\mapsto\gamma\in\C$ and $x\mapsto z$.
Application of Pfaff's $(z\mapsto z/(z-1))$ and Euler's $(z\mapsto z)$ transformations
\cite[\href{http://dlmf.nist.gov/15.8.E1}{(15.8.1)}]{NIST:DLMF} provides the other three representations.
This completes the proof.
\end{proof}
}

\noindent \boro{
The Jacobi function of the second kind
$Q_\gamma^{(\alpha,\beta)}:\C\setminus[-1,1]\to\C$
has the following integral representation
\cite[(4.61.1)]{Szego}
\begin{equation}
Q_\gamma^{(\alpha,\beta)}(z)=\frac{1}{2^{\gamma+1}(z-1)^\alpha(z+1)^\beta}\int_{-1}^{1}\frac{(1-t)^{\alpha+\gamma}(1+t)^{\beta+\gamma}}{(z-t)^{\gamma+1}}{\mathrm d}t,
\label{intrepQ}
\end{equation}
provided $\Re(\alpha+\gamma),\Re(\beta+\gamma)>-1$ \cite[(2.5)]{WimpMcCabeConnor97}.
We will use this integral representation to derive 
an alternative integral 
representation for the Jacobi
function of the second kind.
}

\medskip
{
\begin{thm}
Let $k\in\N_0$, $\alpha,\beta,\gamma,z\in\C$ such that
$z\not\in[-1,1]$. Decompose $\gamma=\delta+k$, where
$\delta=\gamma-k$. Then, the Jacobi function of the second kind
$Q_\gamma^{(\alpha,\beta)}:\C\setminus[-1,1]\to\C$
has the following integral representation
\begin{eqnarray}
&&\hspace{-0.9cm}
Q_\gamma^{(\alpha,\beta)}(z)
\!=\!\frac{\boro{(-1)^k} k!}{\boro{2^{\gamma+1-k}(-\gamma)_k}
(z\!-\!1)^\alpha(z\!+\!1)^\beta}
\int_{-1}^1\!\frac{(1\!-\!t)^{\alpha+\gamma-k}(1\!+\!t)^{\beta+\gamma-k}}
{(z\!-\!t)^{\gamma-k+1}}P_k^{(\alpha+\gamma-k,\beta+\gamma-k)}(t)\,{\mathrm d}t.
\label{intrepQb}
\end{eqnarray}
\end{thm}
}
{
\begin{proof}
Start with 
\eqref{intrepQ} and decompose $\gamma$ such that it is written 
as some complex number $\delta$ added to a non-negative 
integer $k$,
$\gamma=\delta+k$ with
$k\in\N_0$, this produces
\begin{equation}
Q_{\gamma}^{(\alpha,\beta)}(z)=\frac{1}{2^{\delta+k+1}(z-1)^\alpha
(z+1)^\beta}\int_{-1}^{1}\frac{(1-t)^{\alpha+\delta+k}
(1+t)^{\beta+\delta+k}}{(z-t)^{\delta+k+1}}{\mathrm d}t.    
\label{intrepQpn}
\end{equation}
Along the lines of the derivation
of \cite[(4.61.4)]{Szego}
\begin{equation}\label{jf2kconjf1k}
Q_n^{(\alpha,\beta)}(z)
=\frac{1}{2(z-1)^{\alpha}(z+1)^{\beta}}\int_{-1}^1
\frac{(1-t)^\alpha(1+t)^\beta}{z-t}P_n^{(\alpha,\beta)}(t)\,\dd t,
\end{equation}
proceed by integrating \eqref{intrepQpn} by parts
$k$-times with the boundary terms vanishing and utilizing 
the Rodrigues-type formula for Jacobi
polynomials \cite[(9.8.10)]{Koekoeketal}
\begin{equation}
\frac{\dd^k}{\dd t^k}(1-t)^{a+k}(1+t)^{b+k}=(-1)^k2^kk!(1-t)^a(1+t)^bP_k^{(a,b)}(t),
\end{equation}
with $a=\alpha+\delta$, $b=\beta+\delta$. This completes the proof.
\end{proof}
}


The Jacobi function of the second kind
has the following raising and 
lowering operators (see \cite[Theorem 3.1]{IsmailMansour2014}).

\begin{thm}
Let $\gamma,\alpha,\beta\in\C$. Then
\begin{eqnarray}\label{raiOPjFK}
&&\hspace{-2.0cm}\hspace{-5mm}
\frac{\dd}{\dd z}Q_\gamma^{(\alpha,\beta)}(z)
=\frac{\left(\beta-\alpha-z(\alpha+\beta)\right)}{z^2-1}Q_\gamma^{(\alpha,\beta)}(z)-\frac{2(\gamma+1)
}{z^2-1} Q_{\gamma+1}^{(\alpha-1,\beta-1)}(z)\\
&&\hspace{-2.0cm}\hspace{1.6cm}\label{lowOPjFK} 
=-
\frac{\alpha+\beta+\gamma+1}2 Q_{\gamma-1}^{(\alpha+1,\beta+1)}(z).
\end{eqnarray}
\end{thm}
\begin{proof}
We extend the proof given in \cite[Theorem 3.1]{IsmailMansour2014} by replacing integer $n$  by complex $\gamma$ for the 
Jacobi functions of the second kind. 
\end{proof}

\boro{
\begin{thm}\label{mderJacF2K-1}Let $n\in\N_0$, $\alpha, \beta, 
\gamma, z\in\C$. Then
\begin{equation}
\label{derJac2K}
\frac{{\mathrm d}^n}{{\mathrm d} z^n}
(z-1)^\alpha(1+z)^\beta Q_\gamma^{(\alpha,\beta)}(z)=
(-2)^n (\gamma+1)_n (z-1)^{\alpha-n} (1+z)^{\beta-n}
Q_{\gamma+n}^{(\alpha-n,\beta-n)}(z).
\end{equation}
\end{thm}
}
\boro{
\begin{proof}
We start with the $n=1$ case. Taking into account 
\eqref{lowOPjFK} and \eqref{raiOPjFK} one obtains
\[\begin{split}
(1-z)^{-\alpha+1}&(1+z)^{-\beta+1}\frac{{\mathrm d}}{{\mathrm d}z}
(1-z)^\alpha(1+z)^\beta Q_\gamma^{(\alpha,\beta)}(z)=
(1-z^2)\frac{{\mathrm d}}{{\mathrm d} z}Q_\gamma^{(\alpha,\beta)}(z)
\\&+(\beta-\alpha)-(\alpha+\beta)z Q_\gamma^{(\alpha,\beta)}(z)=2(\gamma+1)
Q_{\gamma+1}^{(\alpha-1,\beta-1)}(z).
\end{split}\]
From the above expression, a straightforward calculation allows us to obtain 
the general $n$ integral case.
\end{proof}
}
\boro{
Starting from Theorem \ref{mderJacF2K-1} we can 
obtain a multi-derivative representation
of the Jacobi function of the second kind.
}
{
\begin{cor}
Let $n\in\N_0$, $\alpha,\beta,\gamma\in\C$ such that
$\alpha+\gamma,\beta+\gamma\not\in-\N$. Then
\begin{eqnarray}
Q_\gamma^{(\alpha,\beta)}(z)
=\frac{1}{2^n(-\gamma)_n(z-1)^\alpha(z+1)^\beta}
\frac{\dd^n}{\dd z^n}(z-1)^{\alpha+n}(z+1)^{\beta+n}
Q_{\gamma-n}^{(\alpha+n,\beta+n)}(z).
\end{eqnarray}
\end{cor}
}
\boro{
\begin{proof}
Starting with Theorem \ref{mderJacF2K-1} and mapping
$(\gamma,\alpha,\beta)\mapsto(\gamma-n,\alpha+n,\beta+n)$
completes the proof.
\end{proof}
}

\section{{Multi-integral representations 
for the Jacobi functions}}
\label{miJac}

In this section we derive 
multi-integral formulas for the Jacobi functions
of the first and second kind.

\medskip
\noindent \boro{
First we will derive a lemma which will help us compute some of these multi-integrals.
\begin{lem}\label{lem:1.1}
Let $n,r \in\N_0$, $a, x\in \mathbb C$, $\vec\mu\in \mathbb C^r$,
and let $f^{\vec \mu}$ be a function such that
\begin{equation} \label{eq:fmud}
\frac{\dd}{\dd z}f^{\vec\mu}(z)=\lambda_{\vec\mu} f^{{\vec\mu}\pm\vec 1}(z),
\end{equation}
{where $\lambda_{\vec\mu}\in \mathbb C^*$. 
Then, the following identity holds:}
\[
\int_a^x\cdots \int_a^x f^\mu(w)
(\dd w)^n
=\frac{1}{\lambda_{\vec\mu\mp\vec1}\cdots \lambda_{\vec\mu\mp n\vec1}}
\sum_{k=n}^\infty \frac{\lambda_{\vec\mu\mp n\vec1}\cdots \lambda_{\vec\mu
\mp n\vec1\pm(k-1)\vec1} f^{\vec\mu\mp n\vec1\pm k\vec1}(a)(x-a)^k}{k!},
\]
where $\vec1=(1,1, ...,1)\in \mathbb C^r$.
\end{lem}}
{The proof of this result is analogous to the one of 
\cite[Lemma 2.3]{CohlCostasSantos20}.}

\medskip

\subsection{{The Jacobi functions of the first kind}}

{
\begin{thm}\label{mintJacF1K-1}Let $n\in\N_0$, $\alpha,\beta,\gamma\in\C$, such that $\alpha+\gamma\not \in -\mathbb N$, $\Re\alpha$, $\Re\beta>-1$, $(-\gamma)_n\ne 0$,
$z\in\C\setminus(-\infty,-1]$.
Then
\begin{equation}
\int_{z}^1\cdots\int_{z}^1 (1-w)^\alpha(1+w)^\beta
P_\gamma^{(\alpha,\beta)}(w) \,({\mathrm d}w)^n
=\frac{(-1)^n}{2^n(-\gamma)_n}(1-z)^{\alpha+n}(1+z)^{\beta+n} 
P_{\gamma-n}^{(\alpha+n,\beta+n)}(z),
\end{equation}
\end{thm}
}
{
\begin{proof}
Considering the $n=1$ case of \eqref{derJac1} and then 
integrating produces the following definite integral
\[
\int_{z}^1(1-w)^\alpha(1+w)^\beta 
P_\gamma^{(\alpha,\beta)}(w)\,\dd w=
\frac{1}{2\gamma}(1-z)^{\alpha+1}(1+z)^{\beta+1}
P_{\gamma-1}^{(\alpha+1,\beta+1)}(z),
\]
Iterating the above expression $n$-times completes the proof.
\end{proof}
}

\noindent 
The Jacobi function of the first kind
obeys the following multi-derivative identity.
\begin{thm}\label{mderJacF1K-1}Let $n\in\N_0$, $\alpha,\beta,\gamma\in\C$, $z\in\C\setminus(-\infty,-1]$. Then
\begin{equation}
\label{derJac1}
\frac{{\mathrm d}^n}{{\mathrm d} z^n}
(1-z)^\alpha(1+z)^\beta P_\gamma^{(\alpha,\beta)}(z)=
(-2)^n (\gamma+1)_n (1-z)^{\alpha-n} (1+z)^{\beta-n}
P_{\gamma+n}^{(\alpha-n,\beta-n)}(z).
\end{equation}
\end{thm}

\begin{proof}
Start with
\eqref{der9} and
in \eqref{Jac1}, one has  therefore
\begin{equation}
w=\frac{1-z}{2},\quad \frac{\dd}{\dd w}=-2\frac{\dd}{\dd z},
\label{shift2}
\end{equation}
and then we derive
\begin{eqnarray}
&&\hspace{-1.4cm}\frac{\dd^n}{\dd z^n}(1-z)^{c-1}(1+z)^{a+b-c}
\Ohyp21{a,b}{c}{\frac{1-z}{2}}\nonumber\\
&&\hspace{1cm}=(-2)^n(1-z)^{c-n-1}(1+z)^{a+b-c-n}
\Ohyp21{a\!-\!n,b\!-\!n}{c\!-\!n}{\frac{1-z}{2}}.
\label{effimp}
\end{eqnarray}
Substituting the Gauss hypergeometric function in \eqref{Jac1} 
into the above expression completes the proof.
\end{proof}

\begin{thm}\label{mintJacF1K-1-inf}Let $n\in\N_0$, $\alpha,\beta,\gamma\in\C$, such that $\alpha+\gamma\not \in -\mathbb N$, $\Re(\alpha+\beta+\gamma)<-n$, $\Re \gamma>n-1$,
$z\in\C\setminus(-\infty,-1]$.
Then
\begin{equation}\label{mijf1k:zTOinf}
\int_{z}^\infty \cdots\int_{z}^\infty (1-w)^\alpha(1+w)^\beta
P_\gamma^{(\alpha,\beta)}(w) \,({\mathrm d}w)^n
=\frac{1}{2^n(-\gamma)_n}(1-z)^{\alpha+n}(1+z)^{\beta+n} 
P_{\gamma-n}^{(\alpha+n,\beta+n)}(z).
\end{equation}
\end{thm}
\begin{proof}
Considering the $n=1$ case of \eqref{derJac1},
taking into account the constraints considered in the statement 
and then integrating produces the following improper integral:
\[
\int_{z}^\infty (1-w)^\alpha(1+w)^\beta P_\gamma^{(\alpha,\beta)}(w)\,\dd w=
\frac{1}{2\gamma}(1-z)^{\alpha+1}(1+z)^{\beta+1}
P_{\gamma-1}^{(\alpha+1,\beta+1)}(z),
\]
Iterating the above expression $n$-times completes the proof.
\end{proof}

{
\begin{thm}\label{mderJacF1K-2}Let $n\in\N_0$, $\alpha,\beta,\gamma\in\C$,
$z\in\C\setminus(-\infty,-1]$. 
Then
\begin{equation}
\frac{{\mathrm d}^n}{{\mathrm d} z^n}
(1-z)^\alpha P_\gamma^{(\alpha,\beta)}(z)=
 (-\alpha-\gamma)_n (1-z)^{\alpha-n} 
P_{\gamma}^{(\alpha-n,\beta+n)}(z).
\label{derJac2}
\end{equation}
\end{thm}
}
{
\begin{proof}
Adopting \eqref{effimp} and utilizing 
the Gauss hypergeometric representation 
\eqref{Jac2} completes the proof.
\end{proof}
}

{
\begin{thm}\label{mintJacF1K-2}Let $n\in\N_0$, $\alpha,\beta,\gamma\in\C$, with $\Re \alpha>-1$,
$z\in\C\setminus(-\infty,-1]$. 
Then
\begin{equation}
\int_{z}^1\cdots\int_{z}^1 (1-w)^\alpha
P_\gamma^{(\alpha,\beta)}(w) \,({\mathrm d}w)^n
=\frac{(1-z)^{\alpha+n}}{(\alpha+\gamma+1)_n} 
P_{\gamma}^{(\alpha+n,\beta-n)}(z).
\end{equation}
\end{thm}
}
{
\begin{proof}
Considering the $n=1$ case of \eqref{derJac2} and then 
integrating produces the following definite integral
\[
\int_{z}^1(1-w)^\alpha P_\gamma^{(\alpha,\beta)}(w)\,\dd w=
\frac{(1-z)^{\alpha+1}}{\alpha+\gamma+1}
P_{\gamma}^{(\alpha+1,\beta-1)}(z).
\]
Iterating the above expression $n$-times completes the proof.
\end{proof}
}
{
{
\begin{thm}\label{mintJacF1K-2-inf}Let $n\in\N_0$, $\alpha,\beta,\gamma\in\C$, with $\Re(\alpha+\gamma)<-n$,
$\Re(\beta+\gamma)>n-1$, $z\in\C\setminus(-\infty,-1]$. 
Then
\begin{equation} \label{mijf1k:zTOinf-2}
\int_{z}^\infty\cdots\int_{z}^\infty (1-w)^\alpha
P_\gamma^{(\alpha,\beta)}(w) \,({\mathrm d}w)^n
=\frac{(1-z)^{\alpha+n}}{(\alpha+\gamma+1)_n} 
P_{\gamma}^{(\alpha+n,\beta-n)}(z).
\end{equation}
\end{thm}
}
{
\begin{proof}
Considering the $n=1$ case of \eqref{derJac2},
taking into account the constraints considered in the statement 
and then integrating produces the following improper integral:
\[
\int_{z}^\infty(1-w)^\alpha P_\gamma^{(\alpha,\beta)}(w)\,\dd w=
\frac{(1-z)^{\alpha+1}}{\alpha+\gamma+1}
P_{\gamma}^{(\alpha+1,\beta-1)}(z).
\]
Iterating the above expression $n$-times completes the proof.
\end{proof}
}
}
{
\begin{thm}\label{mderJacF1K-2b}
Let $n\in\N_0$, $\alpha,\beta,\gamma\in\C$,
$z\in\C\setminus(-\infty,-1]$. Then
\begin{equation}
\frac{{\mathrm d}^n}{{\mathrm d} z^n}
(1+z)^\beta P_\gamma^{(\alpha,\beta)}(z)=
(-1)^n (-\beta-\gamma)_n (1+z)^{\beta-n} 
P_{\gamma}^{(\alpha+n,\beta-n)}(z).
\label{derJac2b}
\end{equation}
\end{thm}
}
{
\begin{proof}
Adopting \eqref{effimp} and utilizing 
the Gauss hypergeometric representation 
\eqref{Jac2} completes the proof.
\end{proof}
}

{
{
\begin{thm}\label{mintJacF1K-2b-inf}Let $n\in\N_0$, $\alpha,\beta,\gamma\in\C$, with $\Re(\beta+\gamma)<-n$,
$\Re(\alpha+\gamma)>n-1$, $z\in\C\setminus(-\infty,-1]$. 
Then
\begin{equation} \label{mijf1k:zTOinf-2b}
\int_{z}^\infty\cdots\int_{z}^\infty (1+w)^\beta
P_\gamma^{(\alpha,\beta)}(w) \,({\mathrm d}w)^n
=\frac{(-1)^n(1+z)^{\beta+n}}{(\beta+\gamma+1)_n} 
P_{\gamma}^{(\alpha-n,\beta+n)}(z).
\end{equation}
\end{thm}
}
{
\begin{proof}
Considering the $n=1$ case of \eqref{derJac2b},
taking into account the constraints considered in the statement 
and then integrating produces the following improper integral:
\[
\int_{z}^\infty(1+w)^\beta P_\gamma^{(\alpha,\beta)}(w)\,\dd w=
\frac{-(1+z)^{\beta+1}}{\beta+\gamma+1}
P_{\gamma}^{(\alpha-1,\beta+1)}(z).
\]
Iterating the above expression $n$-times completes the proof.
\end{proof}
}
}
{
\begin{thm}\label{mderJacF1K-3}Let $n\in\N_0$, $\alpha,
\beta,\gamma\in\C$,
$z\in\C\setminus(-\infty,-1]$. Then
\begin{equation}
\frac{{\mathrm d}^n}{{\mathrm d} z^n}
 P_\gamma^{(\alpha,\beta)}(z)=
 2^{-n}(\alpha+\beta+\gamma+1)_n 
P_{\gamma-n}^{(\alpha+n,\beta+n)}(z).
\label{derJac3}
\end{equation}
\end{thm}
}
{
\begin{proof}
Start with \eqref{der2}
and using \eqref{Jac1} with \eqref{shift2} completes the proof.
\end{proof}
}
\noindent 
{
\begin{thm}\label{mintJacF1K-3}Let $n\in\N_0$, $\alpha,\beta,\gamma\in\C$,
with $\alpha+\gamma\not \in -\mathbb N$, $z\in\C\setminus(-\infty,-1]$. Then
\begin{align}
\int_{z}^{1}\cdots\int_{z}^1 &
P_\gamma^{(\alpha,\beta)}(w) \,({\mathrm d}w)^n
=\frac{2^n}{(-\alpha\!-\!\beta\!-\!\gamma)_n} 
P_{\gamma+n}^{(\alpha-n,\beta-n)}(z)\nonumber\\[0.1cm]
&+\frac{2\Gamma(\alpha+\gamma+1)(1-z)^{n-1}}{(n-1)!(\alpha+\beta+\gamma)\Gamma(\alpha)
\Gamma(\gamma+2)}\hyp32{1\!-\!n,1\!-\!\alpha,1}
{\gamma\!+\!2,1\!-\!\alpha\!-\!\beta\!-\!\gamma}{\frac{2}{1-z}}\\[3mm]
&=\frac{\Gamma(\alpha+\gamma+1)(1-z)^n}{\Gamma(\alpha+1)\Gamma(\gamma+1)n!}
\hyp32{-\gamma,\alpha+\beta+\gamma+1,1}{\alpha+1,n+1}{\frac{1-z}2}.
\end{align}
\end{thm}
}
{
\begin{proof}
Considering the $n=1$ case of \eqref{derJac3} and then integrating using \eqref{Jacone} 
produces the following definite integral
\[
\int_{z}^{1} P_\gamma^{(\alpha,\beta)}(w)\,\dd w=
\frac{2}{\alpha+\beta+\gamma}\left(
\frac{\Gamma(\alpha+\gamma+1)}{\Gamma(\alpha)\Gamma(\gamma+2)}-
P_{\gamma+1}^{(\alpha-1,\beta-1)}(z)\right).
\]
Iterating the above expression $n$-times, reversing the order of the sum and utilizing 
standard properties of Pochhammer symbols such as 
\eqref{Pochdef}--\eqref{Poch2},
results in the term involving the
terminating generalized hypergeometric  ${}_3F_2$ series. The term involving the 
Jacobi function of the first kind is clear.
{The second identity is a direct consequence of Lemma \ref{lem:1.1}}.
This completes the proof.
\end{proof}
}
\begin{rem}
Taking into account lemma \ref{lem:1.1} and the previous result 
and setting $\alpha\mapsto \alpha+n$, $\beta\mapsto\beta+n$, 
$\gamma\mapsto\gamma-n$ one obtains the following identity for the 
Taylor series of the Jacobi function of order $n-1$ at $z=1$:
\begin{equation}\label{tayserJf1K}\begin{split}
\sum_{k=0}^{n-1} \frac {d^k}{dz^k}P^{(\alpha,\beta)}_{\gamma}(z)|_{z=1}
\frac{(z-1)^k}{k!}=&
\frac{\Gamma(\alpha+\gamma+1)\Gamma(-\alpha-\beta-\gamma)}
{\Gamma(\alpha+n)(n-1)!}\left(\frac{z-1}2\right)^{n-1}\\ &\times 
\Ohyp32{1-n,1-\alpha-n,1}{\gamma-n+2,1-\alpha-\beta-\gamma-n}{\frac 2{1-z}}.
\end{split}\end{equation}
\end{rem}
{
{
\begin{thm}\label{mintJacF1K-3-inf}Let $n\in\N_0$, $\alpha,\beta,\gamma\in\C$, with $\Re\gamma<-n$,
$\Re(\alpha+\beta+\gamma)>n-1$, $z\in\C\setminus(-\infty,-1]$. 
Then
\begin{equation} \label{mijf1k:zTOinf-3}
\int_{z}^\infty\cdots\int_{z}^\infty (1+w)^\beta
P_\gamma^{(\alpha,\beta)}(w) \,({\mathrm d}w)^n
=\frac{2^n}{(-\alpha\!-\!\beta\!-\!\gamma)_n} 
P_{\gamma+n}^{(\alpha-n,\beta-n)}(z).
\end{equation}
\end{thm}
}
{
\begin{proof}
Considering the $n=1$ case of \eqref{derJac3},
taking into account the constraints considered in the statement 
and then integrating produces the following improper integral:
\[
\int_{z}^\infty P_\gamma^{(\alpha,\beta)}(w)\,\dd w=
\frac{-2}{\alpha+\beta+\gamma} 
P_{\gamma+1}^{(\alpha-1,\beta-1)}(z).
\]
Iterating the above expression $n$-times completes the proof.
\end{proof}
}
}
If we consider the different hypergeometric 
representations for the Jacobi function of the 
first kind \eqref{Jac1}, \eqref{Jac2}, \eqref{Jac3}, \eqref{Jac4} and 
applying the derivative relations 
\eqref{der2}, \eqref{der4}, \eqref{der6},
\eqref{der9} one obtains the following result.
\begin{thm}\label{thm:3.12}
Let $n\in \mathbb N_0$, $\alpha,\beta,\gamma\in 
\mathbb C$, with $\alpha+\gamma\not \in -\mathbb N$, 
$z\in\C\setminus(-\infty,-1]$. The following identities hold:
\begin{eqnarray}
&&\hspace{-5mm}\left[(z-1)^2\frac{{\mathrm d}}{{\mathrm d} z}\right]^n (z-1)^{\alpha+\beta+\gamma+1}P_\gamma^{(\alpha,\beta)}(z)=
(\alpha+\beta+\gamma+1)_n
(z-1)^{\alpha+\beta+\gamma+1+n}
P_\gamma^{(\alpha,\beta+n)}(z),\label{relmd1} \\
&&\hspace{-5mm}\left[(z-1)^2\frac{{\mathrm d}}{{\mathrm d} z}\right]^n \frac{1}{(z-1)^{\gamma}} P_\gamma^{(\alpha,\beta)}(z)=
\frac{(-\alpha-\gamma)_n}{(z-1)^{\gamma-n}}
P_{\gamma-n}^{(\alpha,\beta+n)}(z),\label{relmd2}\\
&&\hspace{-5mm}\left[(z-1)^2\frac{{\mathrm d}}{{\mathrm d} z}\right]^n (z+1)^\beta(z-1)^{\alpha+\gamma+1} P_\gamma^{(\alpha,\beta)}(z)\nonumber\\[-0.05cm]
&&\hspace{20mm}
=2^n(\gamma+1)_n
(z+1)^{\beta-n} (z-1)^{\alpha+\gamma+1+n} P_{\gamma+n}^{(\alpha,\beta-n)}(z),\label{relmd3}\\[0.15cm]
&&\hspace{-5mm}\left[(z-1)^2\frac{{\mathrm d}}{{\mathrm d} z}\right]^n \frac{(z+1)^\beta}{(z-1)^{\beta+\gamma}} P_\gamma^{(\alpha,\beta)}(z)=2^n(-\beta-\gamma)_n
\frac{(z+1)^{\beta-n}}{(z-1)^{\beta-n+\gamma}}
P_{\gamma}^{(\alpha,\beta-n)}(z).\label{relmd4}
\end{eqnarray}
\end{thm}
\begin{proof}
First consider the $n=1$ case. 
Start with the representation \eqref{Jac1} multiply it by 
$(z-1)^{\alpha+\beta+\gamma+1}$ and apply the derivative 
relation \eqref{der2}, to obtain
\[
\frac{{\mathrm d}}{{\mathrm d} z} (z-1)^{\alpha+\beta+\gamma+1} P_\gamma^{(\alpha,\beta)}(z)=(\alpha+\beta+\gamma+1)
(z-1)^{\alpha+\beta+\gamma}
P_\gamma^{(\alpha,\beta+1)}(z).
\]
Multiplying the expression by $(z-1)^2$ and iterating the identity produces 
\eqref{relmd1}. 
By repeating an analogous process for the expressions \eqref{Jac2}, \eqref{Jac3},  \eqref{Jac4}, 
the result follows.
\end{proof}
As a consequence of this result, we have the 
following derivative relations.
\begin{cor}\label{cor:3.16}
Let $n\in \mathbb N_0$, $\alpha,\beta,\gamma, z\in 
\mathbb C$. The following identities hold:
\begin{eqnarray}
&&\hspace{-3mm}\left[(z+1)^2\frac{{\mathrm d}}{{\mathrm d} z}\right]^n (z+1)^{\alpha+\beta+\gamma+1}P_\gamma^{(\alpha,\beta)}(z)=
(\alpha+\beta+\gamma+1)_n
(z+1)^{\alpha+\beta+\gamma+1+n}
P_\gamma^{(\alpha+n,\beta)}(z),\label{relmd5}\\
&&\hspace{-3mm}\left[(z+1)^2\frac{{\mathrm d}}{{\mathrm d} z}\right]^n 
\frac{1}{(z+1)^{\gamma}} P_\gamma^{(\alpha,\beta)}(z)=
\frac {(1+\beta+\gamma-n)_n}{(z+1)^{\gamma-n}}
P_{\gamma-n}^{(\alpha+n,\beta)}(z),\label{relmd6}\\
&&\hspace{-3mm}\left[(z+1)^2\frac{{\mathrm d}}{{\mathrm d} z}\right]^n (z-1)^\alpha(z+1)^{\beta+\gamma+1} P_\gamma^{(\alpha,\beta)}(z)\nonumber \\
&&\hspace{32mm}=2^n(\gamma+1)_n
(z-1)^{\alpha-n} (z+1)^{\beta+\gamma+n}
P_{\gamma+n}^{(\alpha-n,\beta)}(z),\label{relmd7}\\
&&\hspace{-3mm}\left[(z+1)^2\frac{{\mathrm d}}{{\mathrm d} z}\right]^n \frac{(z-1)^\alpha}{(z+1)^{\alpha+\gamma}} P_\gamma^{(\alpha,\beta)}(z)=(-2)^n(-\alpha-\gamma)_n
\frac{(z-1)^{\alpha-n}}{(z+1)^{\alpha-n+\gamma}}
P_{\gamma}^{(\alpha-n,\beta)}(z).\label{relmd8}
\end{eqnarray}
\end{cor}
\begin{proof}
Starting using \eqref{der9} with
\begin{equation}
w=\frac{z-1}{z+1}, \quad 
\frac{\dd}{\dd w}=\frac{(z+1)^2}{2}\frac{\dd}{\dd z},
\label{shift3}
\end{equation}
and substituting in \eqref{Jac1}, \eqref{Jac2}, 
\eqref{Jac3} and \eqref{Jac4} in an analogous 
way as we did in the Theorem \ref{thm:3.12} one obtains 
the expressions.
This completes the proof.
\end{proof}

\begin{cor}[Rodrigues-type formula]
Let $n\in\N_0$, $\alpha,\beta,z\in\C$. 
The Jacobi polynomial admits the following Rodrigues-type 
formula:
\begin{align}\label{rodfjf1k-1}
P_{n}^{(\alpha,\beta)}(z)=&\frac{1}{2^n n!}\frac{1}
{(z-1)^{\alpha}(z+1)^{\beta+n+1}}
\left[(z+1)^2\frac{{\mathrm d}}{{\mathrm d} z}\right]^n\!
 (z-1)^{\alpha+n}(z+1)^{\beta+1}\\
 \label{rodfjf1k-2}
=&\frac{1}{2^n n!}\frac{1}
{(z-1)^{\alpha+n+1}(z+1)^{\beta}}
\left[(z-1)^2\frac{{\mathrm d}}{{\mathrm d} z}\right]^n\!
 (z-1)^{\alpha}(z+1)^{\beta+n+1}.
\end{align}
\end{cor}
\begin{proof}
Setting $\gamma\mapsto 0$ and $\beta\mapsto \beta+n$ in 
\eqref{relmd3} the first identity follows. 
Setting $\gamma\mapsto 0$ and $\alpha\mapsto \beta+n$ in 
\eqref{relmd7} the second identity follows. 
\end{proof}
\begin{thm}
{Let $n\in\N_0$, $\alpha,\beta,\gamma\in \C$,
with $\alpha+\gamma\not \in -\mathbb N$, 
$z\in\C\setminus(-\infty,-1]$. 
The following multi-integrals hold:
\begin{eqnarray}
&&\hspace{-20mm}\int_1^z\cdots\int_1^z (w-1)^{\alpha+\beta+\gamma+1}
P_\gamma^{(\alpha,\beta)}(w)[(w-1)^{-2}\,{\mathrm d}w]^n\nonumber\\
&&\hspace{30mm}\label{relmi1-E}=\frac {(z-1)^{\alpha+\beta+\gamma+1-n}}
{(\alpha+\beta+\gamma-n+1)_{n}}P_\gamma^{(\alpha,\beta-n)}(z),\\
&&\hspace{-20mm}\int_1^z\cdots\int_1^z (w+1)^{\beta} (w-1)^{\alpha+\gamma+1}
P_{\gamma}^{(\alpha,\beta)}(w)[(w-1)^{-2}\,{\mathrm d}w]^n\nonumber\\
&&\hspace{30mm}=\frac{(z+1)^{\beta+n}(z-1)^{\alpha+\gamma-n+1}}{2^n(\gamma-n+1)_n}
P_{\gamma-n}^{(\alpha,\beta+n)}(z),\label{relmi3-E}\\
&&\hspace{-20mm}\int_1^z\cdots\int_1^z 
(w-1)^{\alpha}(w+1)^{\beta+\gamma+1}P_\gamma^{(\alpha,\beta)}(w) 
[(w+1)^{-2}\,{\mathrm d}w]^n\nonumber\\
&&\hspace{30mm} =\frac{(z-1)^{\alpha+n}(z+1)^{\beta+\gamma-n+1}}
{2^n(\gamma-n+1)_n} P_{\gamma-n}^{(\alpha+n,\beta)}(z)\label{relmi7-E},\\
&&\hspace{-20mm}\int_1^z\cdots\int_1^z  \frac{(w-1)^\alpha}{(w+1)^{\alpha+\gamma}} 
P_\gamma^{(\alpha,\beta)}(w)[(w+1)^{-2}\,{\mathrm d}w]^n\nonumber\\
&&\hspace{30mm}=\frac{(z-1)^{\alpha+n}}{2^n(1+\alpha+\gamma)_n(z+1)^{\alpha+n+\gamma}}
P_{\gamma}^{(\alpha+n,\beta)}(z),\label{relmi8-E}
\end{eqnarray}
where $\Re(\alpha+\beta+\gamma+1)>n$, 
$\Re(\alpha+\gamma+1)>n$, 
$\Re(\alpha+n)>0$, $\Re(\alpha+n)>0$, respectively.
}
\end{thm}
\begin{proof}
Consider the $n=1$ case of \eqref{relmd1} and then integrate both sides 
using the fundamental theorem of calculus. 
This produces the following definite integral
\[
\int_1^z (w-1)^{\alpha+\beta+\gamma+1}
P_\gamma^{(\alpha,\beta)}(w)\, (w-1)^{-2}\,{\mathrm d}w=
\frac {(z-1)^{\alpha+\beta+\gamma+1}}
{\alpha+\beta+\gamma}P_\gamma^{(\alpha,\beta-1)}(z).
\]
Iterating the above expression $n$-times completes the proof. 
The process for the remaining cases, i.e., for the cases starting with 
\eqref{relmd3}, \eqref{relmd7} and \eqref{relmd8}, 
is similar so we will omit their proofs. Hence the result holds.
\end{proof}

\begin{thm}\label{thm:3.19}
Let $n\in\N_0$, $\alpha,\beta,\gamma\in \C$,
such that $\alpha+\gamma\not \in -\mathbb N$, 
$z\in\C\setminus(-\infty,-1]$.
The following multi-integrals hold:
\begin{align}
\int_{1}^{z}\cdots\int_{1}^{z} 
&\frac {P_\gamma^{(\alpha,\beta)}(w)}{(w-1)^{\gamma}}
[(w-1)^{-2}\,{\mathrm d}w]^n=\dfrac{(-1)^n}{(\alpha+\gamma+1)_n}
\dfrac{P_{\gamma+n}^{(\alpha,\beta-n)}(z)}{(z-1)^{\gamma+n}}
,\label{relmi2-E}\\[3mm]
\int_{1}^{z}\cdots\int_{1}^{z}& \dfrac {(w+1)^\beta}{(w-1)^{\gamma+\beta}}
P_\gamma^{(\alpha,\beta)}(w) [(w-1)^{-2}\,{\mathrm d}w]^n
=\dfrac{(-1)^n}{2^n(\beta+\gamma+1)_n}\dfrac 
{(z+1)^\beta+n}{(z-1)^{\gamma+\beta}}
P_\gamma^{(\alpha,\beta)}(z)
,\label{relmi4-E}
\end{align}
where $\Re\gamma<-n$ and $\Re(\gamma+\beta)<0$, respectively.
\end{thm}
The proof is analogous to those carried out previously, and we 
leave it to the reader
\begin{thm}
Let $n\in\N_0$, $\alpha,\beta,\gamma\in \C$,
such that $\alpha+\gamma\not \in -\mathbb N$, 
$z\in\C\setminus(-\infty,-1]$.
The following multi-integrals hold:
\begin{eqnarray}
&&\hspace{-0.5cm}\int_{1}^{z}\cdots\int_{1}^{z} 
(w+1)^{\alpha+\beta+\gamma+1}P_\gamma^{(\alpha,\beta)}(w)
[(w+1)^{-2}\,{\mathrm d}w]^n=\frac{(z+1)^{\alpha+\beta+\gamma+1-n}
P_{\gamma}^{(\alpha-n,\beta)}(z)}{(\alpha+\beta+\gamma+1-n)_n}
\nonumber\\[1mm]
&&\hspace{0.5cm}-\frac{2^{\alpha+\beta+\gamma+1-n}\Gamma(\alpha+\gamma)}
{(\alpha+\beta+\gamma)\Gamma(\alpha, \gamma+1,n)}\left(\frac{z-1}{z+1}\right)^{n-1}\!
\hyp32{-n+1,1-\alpha,1}{1-\alpha-\gamma,1-\alpha-\beta-\gamma}
{\frac{z+1}{z-1}},
\label{relmi5-E}\\[3mm]
&&\hspace{-0.5cm}\int_{1}^{z}\cdots\int_{1}^{z}\frac{1}{(w+1)^{\gamma}}
P_\gamma^{(\alpha,\beta)}(w)[(w+1)^{-2}\,
{\mathrm d}w]^n=\dfrac{P_{\gamma+n}^{(\alpha-n,\beta)}(z)}
{(\beta+\gamma+1)_n(z+1)^{\gamma+n}}\nonumber\\
&&\hspace{0.5cm}
-\frac{\Gamma(\alpha+\gamma+1)}
{2^{\gamma+n}(\beta+\gamma+1)\Gamma(\alpha, \gamma+2,n)}
\left(\frac{z-1}{z+1}\right)^{n-1}\!
\hyp32{-n+1,1-\alpha,1}{2+\gamma,2+\beta+\gamma}{\frac{z+1}{z-1}}.
\label{relmi6-E}
\end{eqnarray}
\end{thm}
\begin{proof}
Considering the $n=1$ case of \eqref{relmd6} and then integrating produces the 
definite integral
\[
\int_{1}^z (w+1)^{\alpha+\beta+\gamma+1}P_\gamma^{(\alpha,\beta)}(w)
\frac {{\mathrm d}w}{(w+1)^{2}}=\frac{(w+1)^{\alpha+\beta+\gamma}
P_\gamma^{(\alpha-1,\beta)}(z)
-2^{\alpha+\beta+\gamma}P_\gamma^{(\alpha-1,\beta)}(1)}
{\beta+\gamma+1},
\]
and due to \eqref{Jacone}
the result follows for such a case. 
Iterating the above expression $n$-times completes the proof, using \eqref{Jacone} as
well as the identities \eqref{Pochdef}, \eqref{Poch1}, taking into account 
\[
\int_1^z \left(\frac{z-1}{z+1}\right)^k \frac {{\mathrm d}w}{(w+1)^{2}}=
\frac 1{2(k+1)}\left(\frac{z-1}{z+1}\right)^{k+1} \quad k=0, 1, ...,
\]
and reversing the finite series, i.e., for any 
non-negative integer $m$
\[
\sum_{k=0}^m \frac{(a_1, ...,a_{r+1})_k}{(b_1,...,b_r)_k}\frac{z^k}{k!}=
\frac{(a_1, ...,a_{r+1})_{m}}{(b_1,...,b_r)_m}\frac{z^m}{m!}\hyp{r+2}{r+1}
{-m,1-m-b_1,...,1-m-b_r,1}
{1-m-a_1,...,1-m-a_{r+1}}{\frac1z},
\]
the result follows.
The proof of the other integral is analogous and we omit its proof.
\end{proof}
\subsection{The Jacobi functions of the second kind}
In this section we derive some multi-integrals for the Jacobi function of the second kind. Since both the Jacobi function of the first kind and the Jacobi function of the second kind are 
 strongly connected  (see
\eqref{jf2kconjf1k}), 
we expect to obtain similar 
multi-integrals to those obtained in the previous section.

\begin{thm}\label{mintJacF2K-2}Let $n\in\N_0$, $\alpha,\beta,\gamma \in\C$, 
$z\in \C\setminus[-1,1]$, with $\Re\alpha$, $\Re\beta>-1$, $\Re\gamma>n$. 
Then
\begin{equation}
\int_{z}^\infty\cdots\int_{z}^\infty (w-1)^\alpha(1+w)^\beta
Q_\gamma^{(\alpha,\beta)}(w) \,({\mathrm d}w)^n
=\frac {(z-1)^{\alpha+n}(1+z)^{\beta+n}}{2^n(\gamma-n+1)_n}
Q_{\gamma-n}^{(\alpha+n,\beta+n)}(z).
\end{equation}
\end{thm}
\begin{proof}
Considering the $n=1$ case of \eqref{derJac2K} and then 
integrating produces the following definite integral
\[\begin{split}
\int_{z}^\infty (w-1)^\alpha(1+w)^\beta
Q_\gamma^{(\alpha,\beta)}(w) \,{\mathrm d}w=&
\frac{1}{2\gamma} \lim_{w\to \infty}\left(
(z-1)^{\alpha+1}(1+z)^{\beta+1}
Q_{\gamma+1}^{(\alpha+1,\beta+1)}(z)
\right.\\
&-\left.(w-1)^{\alpha+1}(1+w)^{\beta+1}
Q_{\gamma+1}^{(\alpha+1,\beta+1)}(w)\right)\\
=& \frac {(z-1)^{\alpha+1}(1+z)^{\beta+1}}{2 \gamma} 
Q_{\gamma+1}^{(\alpha+1,\beta+1)}(z).
\end{split}\]
Iterating the above expression $n$-times completes the proof.
\end{proof}
\begin{thm}\label{mderJacF2K-2b}
Let $n\in\N_0$, $\alpha,\beta,\gamma,z\in\C$. Then
\begin{equation}
\frac{{\mathrm d}^n}{{\mathrm d} z^n}
(z-1)^\alpha Q_\gamma^{(\alpha,\beta)}(z)=
(-\alpha-\gamma)_n (z-1)^{\alpha-n} 
Q_{\gamma}^{(\alpha-n,\beta+n)}(z).
\label{derJac2Kb}
\end{equation}
\end{thm}
\begin{proof}
First we prove the $n=1$ case. First consider \eqref{dJsk3} and multiply this expression 
by $(z-1)^\alpha$ and differentiate with respect to $z$. This obtains
\[\begin{split}
\frac{\dd}{\dd z} (z-1)^\alpha Q_\gamma^{(\alpha,\beta)}(z)=&
\frac{\dd}{\dd z}\frac{2^{\alpha+\beta+\gamma}\Gamma(\alpha+\gamma+1)\Gamma(\beta+\gamma+1)}
{(z+1)^{\beta+\gamma+1}}\Ohyp21{\gamma+1,\beta+\gamma+1}{\alpha+\beta+2\gamma+2}
{\frac{2}{1+z}}\\
= &-\frac{2^{\alpha+\beta+\gamma}\Gamma(\alpha+\gamma+1)\Gamma(\beta+\gamma+2)}
{\Gamma(\alpha+\beta+2\gamma+2)(z+1)^{\beta+\gamma+2}}
\sum_{k=0}^\infty \frac{(\gamma+1,\beta+\gamma+2)_k}
{(\alpha+\beta+2\gamma+2)_k k!}\left(\frac 2{1+z}\right)^k\\
=&-\frac{2^{\alpha+\beta+\gamma}\Gamma(\alpha+\gamma+1)\Gamma(\beta+\gamma+2)}
{(z+1)^{\beta+\gamma+1}}\Ohyp21{\gamma+1,\beta+\gamma+2}{\alpha+\beta+2\gamma+2}
{\frac{2}{1+z}}\\
=& -(\alpha+\beta)(z-1)^{\alpha-1} Q_{\gamma}^{(\alpha-1,\beta+1)}(z).
\end{split}\]
The $n$th derivative case is obtained by iterating the above procedure.
\end{proof}
\begin{thm}\label{mintJacF2K-3}Let $n\in\N_0$, $\alpha,\beta,\gamma\in\C$, 
$z\in \C\setminus[-1,1]$,
with $\Re\alpha>-1$, $\Re\beta>n-1$, $\Re(\beta+\gamma+1)>n$. Then
\begin{equation}\begin{split}
\int_{z}^\infty\cdots\int_{z}^\infty (w-1)^\alpha
Q_\gamma^{(\alpha,\beta)}(w) \,({\mathrm d}w)^n
=&\frac {(z-1)^{\alpha+n}}{(\alpha+\gamma+1)_n} 
Q_{\gamma}^{(\alpha+n,\beta-n)}(z).
\end{split}\end{equation}
\end{thm}
\begin{proof}
Considering the $n=1$ case of \eqref{derJac2Kb} and then 
integrating produces the following definite integral
\[\begin{split}
\int_{z}^\infty (w-1)^\alpha
Q_\gamma^{(\alpha,\beta)}(w) \,{\mathrm d}w=&
\frac{1}{\alpha+\gamma+1} 
(z-1)^{\alpha+1}Q_{\gamma}^{(\alpha+1,\beta-1)}(z).
\end{split}\]
Iterating the above expression $n$-times completes the proof.
\end{proof}
\begin{thm}Let $n\in\N_0$, $\alpha,\beta,\gamma,z\in\C$. Then
\begin{equation}
\frac{{\mathrm d}^n}{{\mathrm d} z^n}
 Q_\gamma^{(\alpha,\beta)}(z)=
(-2)^{-n}(\alpha+\beta+\gamma+1)_n 
Q_{\gamma-n}^{(\alpha+n,\beta+n)}(z).
\label{derJac2Kc}
\end{equation}
\end{thm}
\begin{proof}
The $n=1$ case follows from \eqref{lowOPjFK} and for $n>1$, the result is obtained by
iterating the above procedure.
\end{proof}
{\begin{thm}\label{mintJacF2K-1}Let $n\in\N_0$, $\alpha,\beta,\gamma\in\C$, 
$z\in \C\setminus[-1,1]$,
with $\Re\alpha>n-1$, $\Re\beta>n-1$, $\Re(\alpha+\beta+\gamma+1)>n$. Then
\begin{equation}\begin{split}
\int_{z}^\infty\cdots\int_{z}^\infty 
Q_\gamma^{(\alpha,\beta)}(w) \,({\mathrm d}w)^n
=&\frac {2^n}{(\alpha+\beta+\gamma-n+1)_n}Q_{\gamma+n}^{(\alpha-n,\beta-n)}(z).
\end{split}\end{equation}
\end{thm}
\begin{proof}
Considering the $n=1$ case of \eqref{derJac2Kc} and then 
integrating produces the following definite integral
\[\begin{split}
\int_{z}^\infty
Q_\gamma^{(\alpha,\beta)}(w) \,{\mathrm d}w=&
\frac 2{\alpha+\beta+\gamma}Q_{\gamma+1}^{(\alpha-1,\beta-1)}(z).
\end{split}\]
Iterating the above expression $n$-times completes the proof.
\end{proof}
}
{If we consider, as we did in the case of the Jacobi functions of the 
first kind, the different hypergeometric representations for the Jacobi 
function of the second kind \eqref{dJsk1}, \eqref{dJsk4}, \eqref{dJsk3},
\eqref{dJsk2} and applying the derivative relations 
\eqref{der2}, \eqref{der4}, \eqref{der6},
\eqref{der9} one obtains the following result.
\begin{thm}\label{thm:3.27}
Let $n\in \mathbb N_0$, $\alpha,\beta,\gamma\in 
\mathbb C$, with $\alpha+\gamma\not \in -\mathbb N$, 
$z\in\C\setminus[-1,1]$. The following identities hold:
\begin{eqnarray}
&&\hspace{-4mm}\left[(z-1)^2\frac{{\mathrm d}}{{\mathrm d} z}\right]^n (z-1)^{\alpha+\beta+\gamma+1}Q_\gamma^{(\alpha,\beta)}(z)=
(\alpha+\beta+\gamma+1)_n
(z-1)^{\alpha+\beta+\gamma+1+n}
Q_\gamma^{(\alpha,\beta+n)}(z),\label{relmd1-1}\\
&&\hspace{-4mm}\left[(z-1)^2\frac{{\mathrm d}}{{\mathrm d} z}\right]^n \frac{1}{(z-1)^{\gamma}} Q_\gamma^{(\alpha,\beta)}(z)=
\frac {(-\alpha-\gamma)_n}{(z-1)^{\gamma-n}}
Q_{\gamma-n}^{(\alpha,\beta+n)}(z),\label{relmd2-1}\\
&&\hspace{-4mm}\left[(z-1)^2\frac{{\mathrm d}}{{\mathrm d} z}\right]^n \hspace{-2.5mm}(z+1)^\beta(z-1)^{\alpha+\gamma+1} Q_\gamma^{(\alpha,\beta)}(z)\nonumber \\
&&\hspace{30mm}=2^n(\gamma+1)_n
(z+1)^{\beta-n} (z-1)^{\alpha+\gamma+1+n}
Q_{\gamma+n}^{(\alpha,\beta-n)}(z),\label{relmd3-1}\\
&&\hspace{-4mm}\left[(z-1)^2\frac{{\mathrm d}}{{\mathrm d} z}\right]^n \frac{(z+1)^\beta}{(z-1)^{\beta+\gamma}} Q_\gamma^{(\alpha,\beta)}(z)=2^n(-\beta-\gamma)_n
\frac{(z+1)^{\beta-n}}{(z-1)^{\beta-n+\gamma}}
Q_{\gamma}^{(\alpha,\beta-n)}(z).\label{relmd4-1}
\end{eqnarray}
\end{thm}
\begin{proof}
The proof is analogous to the proof of Theorem \ref{thm:3.12}. We leave this to the reader.
\end{proof}
}
\begin{cor}
Let $n\in \mathbb N_0$, $\alpha,\beta,\gamma, z\in 
\mathbb C$. The following identities hold:
\begin{eqnarray}
&&\hspace{-5.1mm}\left[(z+1)^2\frac{{\mathrm d}}{{\mathrm d} z}\right]^n (z+1)^{\alpha+\beta+\gamma+1}Q_\gamma^{(\alpha,\beta)}(z)=
(\alpha+\beta+\gamma+1)_n
(z+1)^{\alpha+\beta+\gamma+1+n}
Q_\gamma^{(\alpha+n,\beta)}(z),\label{relmd5-1}\\
&&\hspace{-5.1mm}\left[(z+1)^2\frac{{\mathrm d}}{{\mathrm d} z}\right]^n 
\frac{1}{(z+1)^{\gamma}} Q_\gamma^{(\alpha,\beta)}(z)=
\frac {(1+\beta+\gamma-n)_n}{(z+1)^{\gamma-n}}
Q_{\gamma-n}^{(\alpha+n,\beta)}(z),\label{relmd6-1}\\
&&\hspace{-5.1mm}\left[(z+1)^2\frac{{\mathrm d}}{{\mathrm d} z}\right]^n (z-1)^\alpha(z+1)^{\beta+\gamma+1} Q_\gamma^{(\alpha,\beta)}(z)\nonumber\\
&& \hspace{30mm}=2^n(\gamma+1)_n
(z-1)^{\alpha-n} (z+1)^{\beta+\gamma+n}
Q_{\gamma+n}^{(\alpha-n,\beta)}(z),\label{relmd7-1}\\
&&\hspace{-5.1mm}\left[(z+1)^2\frac{{\mathrm d}}{{\mathrm d} z}\right]^n \frac{(z-1)^\alpha}{(z+1)^{\alpha+\gamma}} Q_\gamma^{(\alpha,\beta)}(z)=(-2)^n(-\alpha-\gamma)_n
\frac{(z-1)^{\alpha-n}}{(z+1)^{\alpha-n+\gamma}}
Q_{\gamma}^{(\alpha-n,\beta)}(z).\label{relmd8-1}
\end{eqnarray}
\end{cor}
\begin{proof}
The proof is analogous to the proof of 
Corollary \ref{cor:3.16}.
We leave this to the reader.
\end{proof}


\def\cprime{$'$} \def\dbar{\leavevmode\hbox to 0pt{\hskip.2ex \accent"16\hss}d}

\end{document}